\documentclass[12pt]{scrartcl}
 \usepackage{cmap}
 \usepackage{amsmath,amsxtra,amscd,amssymb,latexsym,stmaryrd,amsthm, mathtools, mathrsfs}
  \usepackage{xcolor}

\usepackage{lmodern} 
\usepackage[english]{babel}
\usepackage[utf8]{inputenc}
\usepackage[T1]{fontenc}

\usepackage[colorlinks=true,linkcolor=red!70!black,citecolor=green!70!black, urlcolor=magenta!70!black,backref]{hyperref}

\addtokomafont{title}{\rmfamily\itshape}

\theoremstyle{plain}
\newtheorem{theo}{Theorem}[section]
\newtheorem*{theo*}{Theorem}
\newtheorem{coro}[theo]{Corollary}
\newtheorem{prop}[theo]{Proposition}

\theoremstyle{definition}
\newtheorem{defi}[theo]{Definition}
\newtheorem{rema}[theo]{Remark}

\newcommand*{\dd}%
  {\relax\ifnum\lastnodetype>0\mskip\medmuskip\fi\mathrm{d}}

\newcommand{\wass}{\operatorname{\mathsf{D}}}
\newcommand{\diam}{\operatorname{diam}}
\newcommand{\leb}{m}

\newcommand{\Hol}{\operatorname{\mathcal{H}}}
\newcommand{\holc}{\operatorname{Hol}}
\newcommand{\one}{\boldsymbol{1}}

\title{Mixing speed and stability of SRB measures through optimal transportation}

\author{Houssam Boukhecham \thanks{Univ Paris Est Creteil, Univ Gustave Eiffel, CNRS, LAMA UMR8050, F-94010 Creteil, France} \and Beno\^{\i}t R. Kloeckner \footnotemark[1] }

\begin{document}

\maketitle

\begin{abstract}
It is well-known that the SRB measure of a $C^{1+\alpha}$ Anosov diffeomorphism has exponential decay of correlations with respect to Hölder-continuous observables. We propose a new approach to this phenomenon, based on optimal transport. More precisely, we define a space of measures having absolutely continuous disintegrations with respect to some foliation close to the unstable foliation of the map, endowed with a variant of the Wasserstein metric where mass is only allowed to be transported along the diffeomorphism's stable foliation. We show that this metric is indeed finite on that space, and use that the construction makes the diffeomorphism act as a contraction to deduce two corollaries. First, the SRB measure has exponential decay of correlation with respect to pairs of observable that are only asked to be Hölder-continuous \emph{in the stable, respectively unstable direction}, but can be discontinuous overall. Then, we prove quantitative statistical stability: the map sending a $C^{1+\alpha}$ Anosov diffeomorphism to its SRB measure is locally Hölder-continuous (using the $C^1$ metric for diffeomorphisms and the usual Wasserstein metric for measures).
\end{abstract}

\begin{center}\begin{minipage}{.8\textwidth}\textit{
\textbf{Disclaimer.} This article is an advanced draft: some proofs are somewhat sketchy and the results should therefore be taken with a pinch of salt. Both authors have projects outside academia for the near future, and time did not permit to provide a more complete write-up. We still hope the idea is fruitfull and the proofs detailed enough that the method can be used by others in many cases beyond Anosov diffeomophism.
}\end{minipage}\end{center}

\section{Introduction}

Let $T:M\to M$ be a $C^{1+\alpha}$, topologically mixing Anosov diffeomorphism. The chaotic properties of $T$ make it practically impossible to predict the future of an orbit at any given time, and at the same time enable one to predict accurately how such an orbit distributes over a long enough period of time, for Lebesgue-almost all starting point. This distribution is indeed well-known to be described by the \emph{SRB measure} of $T$, which is also a \emph{physical measure}. The ``accuracy'' of the distribution of orbits is quantified by the decay of correlations, which is exponential for Hölder observables. An important contemporary line of research consists in extending this to various classes of systems enjoying some flavor of non-uniformly hyperbolicity, see e.g. \cite{Liverani1995decay, Baladi1996stochastic, Dolgopyat1998flows, Young1998statistical, Benedicks2000Henon, deCastro2004attractor, Varandas2008correlation, Alves2012statistical, DeSimoi2016mostly, Korepanov2019coupling}.
 The purpose of the present article is to propose a novel method to study the decay of correlations, yielding some new results in the case of Anosov diffeomorphisms and hopefully simple enough to be extended in other settings.

The method is by a coupling argument, akin to the one commonly used in finite-states Markov Chains where one creates two realizations of the chain that are coupled one to another: they run independently until they meet at the same state, then they evolve along the same trajectory. Our idea is to do the same for orbits of $T$, with ``meeting at the same state'' replaced by ''lying not too far away on the same stable leaf''. 

The statement of our main result needs a few definitions which we introduce first.

\subsection{Wasserstein metric and variants}

To state our central result, let us first introduce the \emph{Wasserstein metric} $\wass_1$ associated to the distance function $d(\cdot,\cdot)$ defined on $M$ by any Riemannian metric : to any two probability measures $\mu_1,\mu_2$, it associates the distance
\[ \wass^1(\mu_1,\mu_2) := \inf_{\gamma\in\Gamma(\mu_1,\mu_2)} \int_{M\times M} d(x,y) \dd \gamma(x,y)\]
where $\Gamma(\mu_1,\mu_2)$ is the set of \emph{coupling} (aka \emph{transport plans}) between $\mu_1$ and $\mu_2$, i.e. the set of probability measures $\gamma$ on $M\times M$ having $\mu_1,\mu_2$ as marginals. This last condition can be written $p_{i*}\gamma = \mu_i$ for both $i=1,2$, where $p_i:M\times M\to M$ is the projection to the $i$th factor. ``Kantorovich's duality'' enables one to rewrite this as
\[ \wass^1(\mu_1,\mu_2) = \sup_{f\in\operatorname{Lip}(d,1)} \Big\lvert\int f\dd\mu_1 - \int f\dd \mu_2 \Big\rvert\]
where $\operatorname{Lip}(d,k)$ denotes the space of functions $M\to\mathbb{R}$ that are Lipschitz of constant at most $k$ with respect to the metric $d$.
This metric induces the weak-$*$ topology on the set of probability measure (this claim uses the compactness of $M$).

Now, the same concept can be used for other metrics than $d$. We will combine two variations: first, using $d(\cdot,\cdot)^\beta$ for some $\beta\in(0,1)$, which is also a metric. The corresponding Wasserstein metric is
\[ \wass^\beta(\mu_1,\mu_2) := \inf_{\gamma\in\Gamma(\mu_1,\mu_2)} \int_{M\times M} d(x,y)^\beta \dd \gamma(x,y),\]
it still metrizes the weak-$*$ topology, and the duality will then be against $\beta$-Hölder functions with Hölder constant at most $1$ (it is quite different from $\wass^1(\cdot,\cdot)^\beta$, since small movements of mass are disproportionately penalized compared to large ones). Second, we will use the extended metric $d_s:M\times M\to[0,+\infty]$ which gives the distance along the stable leafs: $d_s(x,y)=\infty$ when $x,y$ are not in the same stable leaf, otherwise $d_s(x,y)$ is the least length of a smooth curve from $x$ to $y$, constrained to stay on the same stable leaf. Then we write
\[ \wass_s^1(\mu_1,\mu_2) := \inf_{\gamma\in\Gamma(\mu_1,\mu_2)} \int_{M\times M} d_s(x,y)\dd \gamma(x,y)\]
which can take the value $+\infty$ but is otherwise well-defined, and has a dual formulation against bounded Borel-measurable function that are Lipschitz \emph{in the stable direction}, i.e. with respect to $d_s$. Last, we combine both variations by introducing for all $\beta\in(0,1]$
\[ \wass_s^\beta(\mu_1,\mu_2) := \inf_{\gamma\in\Gamma(\mu_1,\mu_2)} \int_{M\times M} d_s(x,y)^\beta\dd \gamma(x,y).\]
Again, we have a duality and we introduce the corresponding space $\holc_s^\beta(M,T)$ of bounded, Borel-measurable functions $f:M\to\mathbb{R}$ such that for some $C\ge 0$ and all $x,y\in M$,
\[\lvert f(x) - f(y)\rvert \le C\cdot d_s(x,y)^\beta.\]
The least such $C$ is then denoted by $\holc_s^\beta(f)$, and we use on $\holc_s^\beta(M,T)$ the norm $\lVert\cdot\rVert_{s,\beta} = \lVert\cdot\rVert_\infty + \holc_s^\beta(\cdot)$, making it a Banach algebra.
Then
\[ \wass_s^\beta(\mu_1,\mu_2) = \sup_{f} \Big\lvert\int f\dd\mu_1 - \int f\dd \mu_2 \Big\rvert\]
where the supremum is over all $f\in\holc_s^\beta(M,T)$ such that $\holc_s^\beta(f)\le 1$.

We define similarly $\holc_W^\beta$, where the stable foliation is replaced by any foliation $W$, and $\holc_u^\beta:=\holc_{W^u}^\beta$.

\subsection{Regularly Foliated Measures}

The notion of SRB measure combines two aspects: the absolute continuity of the local disintegrations along the unstable foliation, and the $T$-invariance (the usual definition also includes existence of a positive Lyapunov exponent, which follows here from the Anosov assumption). We will consider non-invariant measures that satisfy the first condition for a foliation close to the unstable foliation of $T$.

To define ``close'', we fix an open cone field $(C_x \subset T_xM)_{x\in M}$ containing a neighborhood of the unstable distribution $E^u$ of $T$~; we will denote by $U_x$ the corresponding neighborhood of $E^u_x$ in the Grassmanian $G(k_u,T_xM)$ where $k_u := \dim E^u$; i.e. whenever $E$ is a $k_u$-dimensional subspace of $T_xM$, we write indifferently $E\in U_x$ or $E\subset C_x$, and we say that $E$ is \emph{tangent to the cone field}. A $k_u$-dimensional submanifold $N$ of $M$ is said to be \emph{tangent to the cone field} when $T_xN\in U_x$ for all $x\in N$, and a foliation is said to be tangent to the cone field whenever its leaves are. The cone field will be chosen small enough that the action of $T$ on the sections of $(U_x)_x$ is contracting (with $E^u$ as its unique fixed point).

We use two parameters $\beta\in(0,1]$ and $K>0$ to bound the regularity of both the foliation and the densities of the disintegration. Foliations to be considered will be tangent to the cone field, and have their leaves written as $\beta$-Hölder graphs with constant at most $K$ in some fixed finite atlas of $M$. Measures to be considered will have absolutely continuous local disintegrations with respect to such a foliation, with $\beta$-Hölder densities, and the log-densities will be asked to have Hölder constant at most $K$.
The set of such probability measures is denoted by $\mathscr{R}_K^\beta$ (the precise definition is given as Section \ref{s:defiR}).

\subsection{Results}

We are now in a position to state our central result.
\begin{theo}\label{t:central}
Let $T:M\to M$ be a $C^{1+\alpha}$ Anosov diffeomorphism. For some $\beta_0\in(0,1)$ and for each $\beta\in(0,\beta_0)$ there exist $K_0,C>0$ and $n_0\in\mathbb{N}$ with the following properties:
\begin{enumerate}
\item $\mathscr{R}_{K_0}^\beta$ contains a unique $T$-invariant measure, which is the SRB measure $\mu_0$ of $T$,
\item $\wass_s^{\beta}(\mu_1,\mu_2) \le C$ for all $\mu_1,\mu_2\in \mathscr{R}_{K_0}^\beta$,
\item $T^{n_0}_*\mathscr{R}_{2K}^\beta\subset \mathscr{R}_{K}^\beta$  for all $K\ge K_0$.
\end{enumerate}
\end{theo}
In particular,  for all $\mu_1,\mu_2\in \mathscr{R}_{K_0}^\beta$ we can pair $\mu_1$ and $\mu_2$ by a couling $\gamma$ that is supported on the set of $(x,y)$ where $x$ and $y$ are in the same stable leaf. Since the stable direction is contracting exponentially, denoting by $\lambda\in(0,1)$ any constant such that $T$ is $\lambda$-contracting along its stable foliation $W^s$, we get
\begin{equation}
\wass_s^{\beta}(T^n_*\mu_1,T^n_*\mu_2) \le C\lambda^{\beta n} \qquad \forall n\in\mathbb{N};
\label{eq:decay}
\end{equation}
and since $\wass^1$ is bounded above by $\diam(M)^{1-\beta} \wass_s^{\beta}$, we also have exponential convergence in the usual Wasserstein metric. Given any $\mu\in \mathscr{R}_K^\beta$, the sequence $(T^n_*\mu)$ will converge to an invariant measure in $\mathscr{R}^\beta_{K_0}$, which must then be the SRB measure of $T$.

The following two corollaries are easily deduced from Theorem \ref{t:central}.
\begin{coro}[Exponential decay of correlations for stable/unstable-Hölder pairs of observables]\label{c:decay}
Let $\mu_0$ be the SRB measure of $T$. For all $\beta\le \beta_0$, there exist $C_\beta\ge 1$ such that for all $n\in\mathbb{N}$, all $f\in \holc_s^\beta$ and all $g\in \holc_u^\beta$,
\[\Big\lvert \int f\circ T^n \cdot g \dd\mu_0 -\int f\dd\mu_0\int g\dd\mu_0 \Big\rvert \le C_\beta\,  \lVert f\rVert_{s,\beta} \,  \lVert g\rVert_{u,\beta} \, \lambda^{\beta n} \]
\end{coro}

While exponential decay of correlations is long known for Anosov diffeomorphisms, this version applies to observables with only partial continuity requirement: $f$, $g$ can exhibit discontinuities in the unstable, respectively stable, direction; the duality appearing between $\holc_u^\beta$ and $\holc_s^\beta$ seems very natural. One could also point out that the rate of decay is directly determined by the contraction factor $\lambda$ of $T$, but beware that $\beta_0$, while constructive, is not made completely explicit in the proof and can in principle be very small.

\begin{coro}[Hölder statistical stability]\label{c:stability}
There exist $\varepsilon>0$, $C'$ and $\beta'$ such that for any $C^{1,\alpha}$ Anosov map $T_1:M\to M$ which is $\varepsilon$-close from $T$ in the $C^1$ topology,
the SRB measures $\mu_0,\mu_1$ of $T$ and $T_1$ are close in the Wasserstein distances:
\[ \wass_1(\mu_0,\mu_1) \le C' \lVert T-T_1\rVert_{C^0}^{\beta'}\]
\end{coro}

Here the main emphasis is on the explicit, Hölder modulus of continuity of the map sending an Anosov diffeomorphism to its SRB measure. Also note that the distance between measure is given in the quite strong Wasserstein metric, and that while we need $T_1$ to be $C^1$-close to $T$, the control we get uses only the uniform distance
\[\lVert T-T_1\rVert_{C^0} := \max_{x\in M} d\big(T(x), T_1(x) \big).\]

\section{Warm-up: expanding maps}\label{s:expanding}

Our goal here is to illustrate one strand of our method in the simplest possible setting by providing a short and elementary proof of the exponential decay of correlation for expanding maps. 

The setting is as follows: $T:\mathbb{T}^1\to\mathbb{T}^1$ is a uniformly expanding circle map of class $C^{1+\alpha}$ for some $\alpha\in(0,1]$. Identify $\mathbb{T}^1=\mathbb{R}/\mathbb{Z}$ with $[0,1)$ and denote the usual circle distance by $d$. Since $T$ is expanding, $\lambda:=1/\inf\lvert T'\rvert \in (0,1)$. Since $T$ is $C^{1+\alpha}$, there exist $H>0$ such that $\lvert T'(x)\rvert \le \lvert T'(y)\rvert e^{Hd(x,y)^\alpha}$ for all $x,y\in\mathbb{T}^1$. Let $\leb$ be the Lebesgue measure; searching for an absolutely continuous invariant probability measure (ACIP), one as usual defines the transfer operator $\mathcal{L}$ by $T_*(\rho \leb) = \mathcal{L}(\rho) m$, so that a non-negative eigenfunction of $\mathcal{L}$ for the eigenvalue $1$ is the density of an ACIP. By the change of variable formula, the transfer operator can be expressed as
\[\mathcal{L}\rho(x) := \sum_{z\in T^{-1}(x)} \frac{\rho(z)}{\lvert T'(z) \rvert}\]
and given any $x,y\in \mathbb{T}^1$, we can order $T^{-1}(x)=:\{x_i:1\le i\le k\}$ and $T^{-1}(y)=:\{y_i:1\le i\le k\}$ (where $k=\lvert\deg T\rvert$) in such a way that $d(x_i,y_i)\le \lambda d(x,y)$ for all $i$. Such an ordering will always be assumed whenever we consider a pair of points and their inverse images. It will be convenient to have $\mathcal{L}$ act on $\alpha$-Hölder function, i.e. we view it as a linear, bounded operator on the Banach algebra $\holc_\alpha(\mathbb{T}^1)$ where the norm is
\[\lVert f\rVert_\alpha := \lVert f\rVert_\infty + \holc_\alpha(f); \qquad \holc_\alpha(f) := \sup \Big\{\frac{\lvert f(x)-f(y)\rvert}{d(x,y)^\alpha} \colon x\neq y\in \mathbb{T}^1 \Big\}.\]

\begin{theo}[Folklore\footnote{By this we mean that there are so many variants of such a result that we prefer not attribute it to a particular article or set of authors.}]\label{t:expanding}
There is an $\alpha$-Hölder positive density $\rho_0:\mathbb{T}^1\to(0,+\infty)$ such that $\rho_0 m$ is $T$-invariant, and there exist $C\ge 1$ and $\theta\in(0,1)$ such that for all $\alpha$-Hölder positive density $\rho$ and all $n\in\mathbb{N}$,
\[\lVert \mathcal{L}^n \rho -\rho_0 \rVert_\alpha \le C \lVert \rho-\rho_0\rVert_\alpha \theta^n.\]
\end{theo}

After the proof, we will briefly recall why this statements contains the decay of correlation, and its relation to a very strong convergence of probability measure, in the total variation norm. This convergence will inspire the metric to be used on measures in the Anosov case.

Our proof starts in a usual way, showing that the ``stretching'' effect stemming from the expansion hypothesis regularizes densities above some determined level. Then we phrase a simple coupling argument in term of densities: regular enough densities stay far from zero, so that any two of them ``share'' a definite amount of mass.

\begin{proof}
For each $K>0$, consider the following set of Hölder densities:
\[\Hol_K^\alpha := \Big\{\rho : \mathbb{T}^1\to(0,+\infty) \,\Big|\, \int\rho \dd\leb = 1, \frac{\rho(x)}{\rho(y)} \le e^{K d(x,y)^\alpha}\ (\forall x,y\in\mathbb{T}^1)\Big\}.\]
For all $\rho\in\Hol^\alpha_K$ and all $x,y\in\mathbb{T}^1$ with (suitably ordered) inverse images $\{x_i\}_i$ and $\{y_i\}_i$: 
\begin{align*}
\frac{\rho}{\lvert T'\rvert}(x) &\le \frac{\rho}{\lvert T'\rvert}(y) \cdot  e^{(K+H)d(x,y)} \\
\frac{\rho}{\lvert T'\rvert} (x_i) &\le \frac{\rho}{\lvert T'\rvert}(y_i)  \cdot e^{(K+H)d(x_i,y_i)^\alpha} \\
\mathcal{L}\rho(x) &\le \mathcal{L}\rho(y) \cdot e^{(K+H)\lambda^\alpha d(x,y)^\alpha}
\end{align*}
so that for all $K>0$, $\mathcal{L}(\Hol^\alpha_K) \subset \Hol^\alpha_{h(K)}$ where $h(K)=(K+H)\lambda^\alpha$. For all $K$ and $n\in\mathbb{N}$:
\begin{align*}
h^n(K) &= K\lambda^{n\alpha} + \lambda^\alpha H + \lambda^{2\alpha} H + \dots +\lambda^{n\alpha} H \\
 &\le K\lambda^{n\alpha} + \frac{\lambda^\alpha}{1-\lambda^\alpha}H.
\end{align*}
Set $K_0=2\frac{\lambda^\alpha}{1-\lambda^\alpha} H$ and let $n_0$ be the smallest integer such that $2\lambda^{n_0\alpha} \le \frac{\lambda^\alpha}{1-\lambda^\alpha}$, to get:
\begin{equation}
\mathcal{L}^{n_0}(\Hol^\alpha_{2K_0}) \subset \Hol^\alpha_{K_0} \qquad\text{and}\qquad 
\mathcal{L}(\Hol^\alpha_{K_0}) \subset \Hol^\alpha_{K_0}
\end{equation}
%
Every $\rho\in \Hol^\alpha_K$, being a density, reaches a value at least $1$ at some point $x_0$, and each $x\in\mathbb{T}^1$ is at distance less than $1$ from $x_0$ so that $\rho(x)\ge e^{-Kd(x,x_0)^\alpha} \ge e^{-K}$.
For all $\tau\in(0,e^{-K_0})$ and all $\rho\in \Hol^\alpha_{K_0}$, we can decompose
\[\rho = \tau \one + (1-\tau)\tilde \rho\]
where $\one$ is the constant function with value $1$ and $\tilde \rho = (\rho-\tau\one)/(1-\tau)$ is positive of integral $1$ with respect to $m$. From now on we fix a value of $\tau$ small enough to further ensure that $\tilde \rho\in\Hol^\alpha_{2K_0}$ whenever $\rho\in\Hol^\alpha_{K_0}$. Then for all $\rho_1,\rho_2\in \Hol^\alpha_{K_0}$ and all $n\ge n_0$:
\begin{align}
\mathcal{L}^{n_0}(\rho_1-\rho_2) 
  &= \mathcal{L}^{n_0}\big(\tau\one+(1-\tau)\tilde\rho_1-\tau\one - (1-\tau)\tilde\rho_2\big) \nonumber\\
  &= (1-\tau) \mathcal{L}^{n_0} (\tilde \rho_1 -\tilde\rho_2) \nonumber\\
\mathcal{L}^{n}(\rho_1-\rho_2)
  &= (1-\tau)\mathcal{L}^{n-n_0} \big(\mathcal{L}^{n_0} \tilde \rho_1 -\mathcal{L}^{n_0} \tilde\rho_2\big) \label{e:contraction}
\end{align}
where $\tilde\rho_i\in \Hol^\alpha_{2K_0}$ and $\mathcal{L}^{n_0} \tilde \rho_i\in \Hol^\alpha_{K_0}$. For each $n\in\mathbb{N}$, define
\[\gamma_n = \sup_{\rho_1,\rho_2\in\Hol^\alpha_{K_0}}  \lVert \mathcal{L}^n(\rho_1-\rho_2) \rVert_\alpha \in [0,+\infty);\]
 the above computation shows
$\gamma_n \le (1-\tau) \gamma_{n-n_0}$
so that $\gamma_n$ goes to $0$, exponentially fast. The sequence of sets of functions $\mathcal{L}^n(\Hol^\alpha_{K_0})$ thus converges to a point $\{\rho_0\}\subset \Hol^\alpha_{K_0}$ in the uniform norm (using compactness provided by the Azelà-Ascoli Theorem). Then $\rho_0$ must be a fixed point of $\mathcal{L}$ and $\rho_0m$ is an invariant measure of $T$. Being a fixed point, $\rho_0$ must belong to $\Hol^\alpha_{K_0/2}$, and $(\mathcal{L}^n\rho)_n$ converges to $\rho_0$ exponentially fast, uniformly over all $\rho\in \Hol^\alpha_{K_0}$. The statement follows.
\end{proof}

\begin{rema}[Decay of correlations]\label{r:decay}
Denote by $\mu_0=\rho_0 m$ the ACIP and let $f,g:\mathbb{T}^1\to\mathbb{R}$ be two ``observables'', with $f\in L^1(m)$  and $g$ $\alpha$-Hölder. For $a$ small enough and $b=1-a\lVert g\rVert_\alpha^{-1}\int g\dd\mu_0$, the function $\rho=(a\lVert g\rVert_\alpha^{-1} g+b)\rho_0$ is a $\alpha$-Hölder density with a $g$-uniform bound on $\lVert \rho-\rho_0\rVert_\alpha$, so that:
\begin{align*}
\int f\circ T^n \cdot g \dd\mu_0 &= \frac{\lVert g\rVert_\alpha}{a}\int f\circ T^n \cdot \big(\frac{\rho}{\rho_0}-b\big) \dd\mu_0 \\
  &= \frac{\lVert g\rVert_\alpha}{a}\Big(\int f \dd T_*^n(\rho m) -b\int f\dd T_*^n\mu_0\Big) \\
  &= \frac{\lVert g\rVert_\alpha}{a}\int f \big(\mathcal{L}^n\rho-\rho_0\big)\dd m +\int g\dd \mu_0 \int f\dd \mu_0 \\
  &=\int g\dd \mu_0 \int f\dd \mu_0 + O(\lVert g\rVert_\alpha \lVert f\rVert_{L^1(m)} \theta^n)
\end{align*}
i.e. we have exponential decay of correlations.
\end{rema}

\begin{rema}
As seen in Remark \ref{r:decay}, Theorem \ref{t:expanding} implies that the measures $T_*^n(\rho m)$ converge to $\mu_0:=\rho_0m$ in duality with $L^1(m)$ test functions; the convergence thus also holds against Borel bounded test functions, i.e. in the \emph{total variation} distance. This can be rephrased as the existence of positive measures $\nu_n$ of total mass increasing to $1$ such that $\nu_n\le  T_*^n(\rho m)$ and $\nu_n\le \mu_0$ for all $n$. Such a strong convergence cannot be expected in the Anosov case, where the measures we will have to work with might be singular with respect to the limit measure. Seeing an expanding map as an Anosov endomorphism with stable dimension $0$ gives a good guess on the natural replacement for the total variation distance: we shall use an optimal transportation distance where mass is only allowed to move along the stable direction. This is the second strand of proof we shall weave with the first one.
\end{rema}

\section{Anosov diffeomorphisms}

\subsection{Coupling measures and deduction of Theorem \ref{t:central}}

For each $\beta\in(0,\alpha_0]$, we shall build a nested family of sets of probability measures $(\mathscr{R}_K^\beta)_K>0$, where $K$ is a regularity bound (the lesser, the more regular), where elements of $\mathscr{R}_K^\beta$ have absolutely continuous local disintegration with respect to some foliation close to the unstable foliation of $T$, and $K$ bounds the Hölder regularity both of the foliation and of the densities on the leafs. For all $L>0$ we define 
\[\Delta_s(L) := \big\{(x,y)\in M\times M \mid d_s(x,y)\le L \big\} \]
the set of pairs of points at distance at most $L$ along the stable foliation.  The slightly technical part of the article (detailed in Section \ref{s:technical}) leads to the following properties:
\begin{prop}\label{p:technical}
There exist $L_0$ and $\tau\in(0,1)$ (depending only on $T$), $K_0>0$ and $n_0\in\mathbb{N}$  (both further depending on $\beta$) such that:
\begin{enumerate}
\item The SRB measure of $T$ lies in $\mathscr{R}_{K}^\beta$ for some $K>0$,
\item $\mathscr{R}_K^\beta\subset\mathscr{R}_{K'}^\beta$ for all $K'>K$,
\item\label{enumi:tech2.5} for all $\mu\in \mathscr{R}_K^\beta$ associated with a foliation $W$ and all positive $\rho\in\holc_W^\beta$ such that $\int \rho\dd\mu=1$, the probability measure  $\rho\mu$ lies in $\mathscr{R}_{K'}^\beta$ where $K'=K + \holc_W^\beta(\log \rho)$,
\item $T_*(\mathscr{R}_{K_0}^\beta)\subset \mathscr{R}_{K_0}^\beta$ and $T^{n_0}_*(\mathscr{R}_{2K}^\beta)\subset \mathscr{R}_{K}^\beta$ for all $K\ge K_0$,
\item\label{enumi:tech4} for all $\mu_1,\mu_2\in \mathscr{R}_{K_0}^\beta$ we can find a probability measure $\eta$ concentrated on $\Delta_s(L_0)$ and two probability measures $\mu'_1,\mu'_2\in \mathscr{R}_{2K_0}$ such that
\[\mu_i=\tau p_{i*}\eta + (1-\tau) \mu'_i \qquad (i\in\{1,2\}).\]
\end{enumerate}
Moreover, for any $\beta\in(0,\alpha_0)$ we can make the map $T\mapsto K_0$ upper semi-continuous in the $C^1$ topology, in particular there exist $\varepsilon>0$ such that to any map $T_1$ that is $\varepsilon$-close to $T$ in the $C^1$ topology we can apply the above with the same $\beta$ and a constant $K_1\le 2K_0$ in the role of $K_0$ (and further constants $L_1, n_1$).
\end{prop}

While item \ref{enumi:tech2.5} might seem difficult to apply, since $W$ might not be known and $\holc_W^\beta$ is thus quite abstract a space, it can be applied e.g. to a $\beta$-Hölder function $\rho$, and in the case of the SRB measure when $W=W_u$.

\begin{proof}[Proof of Theorem \ref{t:central}.]
We first prove that $\wass_s^\beta$ is uniformly bounded over all $\mu_1,\mu_2\in\mathscr{R}_{K_0}^\beta$, and that the SRB measure is the unique $T$-invariant measure in $\mathscr{R}_K^\beta$. For now $\beta$ is fixed anywhere in $(0,\alpha_0]$, and we will see along the proof how we need to further restrict it.

We will apply item \ref{enumi:tech4} of Proposition \ref{p:technical} repeatedly, defining sequences $(\mu_1^k,\mu_2^k)_{k\ge0}$ and $(\eta^k)_{k\ge 0}$ as follows. First, $\mu_i^0=\mu_i\in \mathscr{R}_{K_0}^\beta$ are the measures we start with. Given $\mu_1^k,\mu_2^k$, we apply to them Proposition \ref{p:technical} to get a ``partial coupling'' $\eta^k$ and ``residual measures'' $(\mu_i^k)'\in \mathscr{R}_{2K_0}$. Then we set
\[\mu_i^{k+1} = T^{n_0}_*(\mu_i^k)'\in \mathscr{R}_{K_0}.\]
By induction, for both $i$ and all $k$:
\[T^{kn_0}_*\mu_i = \sum_{j=0}^k  (1-\tau)^j \tau \big[T^{(k-j)n_0}\circ p_{i}\big]_* \eta_j + (1-\tau)^{k+1}(\mu_i^k)' \]
Applying $T^{-kn_0}_*$ to this equality and letting $k\to\infty$, it follows that 
\[\gamma := \tau \sum_{j\ge0}(1-\tau)^j (T^{-jn_0},T^{-jn_0})_*\eta_j\]
is a coupling of $(\mu_1,\mu_2)$. Since $\eta_j$ is concentrated on $\Delta_s(L_0)$, letting $\lambda_0 = \inf\{\lVert D_xT u\rVert\colon x\in M, u\in E^s_x, \lVert u \rVert=1\}$ be the maximal contraction of $T$ on the stable direction, 
\begin{align*}
\int d_s(x,y)^\beta \dd \gamma(x,y) 
  &= \tau \sum_{j\ge 0} (1-\tau)^j \int d_s(T^{-jn_0} x, T^{-jn_0} y)^\beta \dd\eta_j \\
  &\le \tau \sum_{j\ge 0} (1-\tau)^j \int \lambda_0^{-j\beta n_0} d_s( x, y)^\beta \dd\eta_j \\
  &\le \tau L_0^\beta \sum_{j\ge 0} \big(\frac{1-\tau}{\lambda_0^{\beta n_0}}\big)^j
\end{align*}
Set $\beta_0=\min\big\{\alpha_0,\frac{\log(1-\tau)}{n_0\log\lambda_0}\big\}$; then whenever $\beta\in(0,\beta_0)$, the above series is convergent and provides the desired uniform bound.

Last, by Proposition \ref{p:technical} the SRB measure $\mu_0$ lies in some
$\mathscr{R}_{K}^\beta$. It must be that $\mu_0\in \mathscr{R}_{K_0}^\beta$: if $K\le K_0$ this is direct, otherwise it suffices to apply $T$ sufficiently many times and use invariance of $\mu_0$. Uniqueness of the $T$-invariant measure in this set then follows from \eqref{e:contraction}, which only depends on the part of Theorem \ref{t:central} we already proved.
\end{proof}

\subsection{Proof of Corollaries \ref{c:decay} and \ref{c:stability}}

Let $\mu_0$ be the SRB measure of $T$ and consider any fixed $\beta\in(0,\beta_0)$. By Proposition \ref{p:technical}, $\mu_0\in\mathscr{R}_{K_0}^\beta$.
\begin{proof}[Proof of Corollary \ref{c:decay}]
Let $f\in\holc_s^\beta$ and $g\in\holc_u^\beta$. Much like in Remark \ref{r:decay}, we take $a$ small enough (but independent of $g$) and $b=1-a\lVert g\rVert_{u,\beta}^{-1} \int g \dd\mu_0$ ensuring that $\rho:= a\lVert g\rVert_{u,\beta}^{-1} g+b$ is positive, that $\int \rho\dd\mu_0=1$, and that $\rho\mu_0\in \mathscr{R}_{2K_0}^\beta$ (item \ref{enumi:tech2.5} of Proposition \ref{p:technical}). Then
\begin{align*}
\int f\circ T^n \cdot g \dd\mu_0 &= \frac{\lVert g\rVert_\alpha}{a}\int f\circ T^n \cdot (\rho-b) \dd\mu_0 \\
  &= \frac{\lVert g\rVert_\alpha}{a}\Big(\int f \dd T_*^n(\rho \mu_0) -b\int f\dd T_*^n\mu_0\Big) \\
  &= \frac{\lVert g\rVert_\alpha}{a}\Big(\int f \dd T_*^n(\rho \mu_0)- \int f\dd\mu_0\Big) +\int g\dd \mu_0 \int f\dd \mu_0 \\
  &=\int g\dd \mu_0 \int f\dd \mu_0 + O\Big(\lVert g\rVert_{u,\beta} \lVert f\rVert_{s,\beta} \ \wass_s^\beta\big(T_*^n(\rho \mu_0),\mu_0\big)\Big)
\end{align*}
Using Theorem \ref{t:central} and its consequence \eqref{eq:decay}, since $T_*^{n_0}(\rho \mu_0)\in\mathscr{R}_{K_0}^\beta$, for all $n\in\mathbb{N}$:
\begin{align*}
\wass_s^\beta(T_*^{n}(\rho \mu_0),\mu_0) &= \wass_s^\beta(T_*^{n-n_0}(T_*^{n_0}\rho \mu_0)), T_*^{n-n_0}\mu_0)\le C\lambda^{\beta(n-n_0)}.
\end{align*}
\end{proof}

\begin{proof}[Proof of Corollary \ref{c:stability}]
We shall denote with an index $1$ instead of $0$ the quantities related to $T_1$ instead of $T$ (e.g. $\mu_1$, $\alpha_1$, $K_1$, etc.) with implied dependency on $T_1$, $\beta$, etc.

Fix any $\beta<\beta_0$; taking $\varepsilon$ small enough ensures that the unstable distribution $E^u_1$ of $T_1$ is tangent to the cone field fixed around $E^u$ to define the $(\mathscr{R}^\beta_K)_K$, and that $\alpha_1>\beta$. By Proposition \ref{p:technical} we can further assume that $K_1\le 2K_0$. Observe that whenever Proposition \ref{p:technical} holds for some value of $K_0$, it also hold for all larger values. Up to enlarging $K_0$, we can thus assume that $K_1\le K_0$ and that the unstable foliation of $T_1$ is $\beta$-Hölder of constant $K_0$. This enlargement really depends on $\varepsilon$ only, uniformly on all $T_1$. We thus have $\mu_1\in \mathscr{R}_{K_0}^\beta$.

%

Normalize the Riemannian metric to ensure the diameter of $M$ is $\le 1$.
Let $A=\lVert T\rVert_{C^1}+\varepsilon$, so that both $T$ and $T_1$ are $A$-Lipschitz.
Set $D= \wass_1(\mu_0,\mu_1)$, and fix $n\in\mathbb{N}$ large enough to ensure that $\wass_\beta^s(T^n_*\mu_1,T^n_*\mu_0)\le \frac12 \wass_1(\mu_0,\mu_1)$; we can take $n\simeq \log 1/D$, so that $A^n \simeq D^{-p}$ for some $p>0$ independent of $T_1$.
Then
\begin{align*}
D = \wass_1(\mu_1,\mu_0) &\le \wass_1(\mu_1,T^n_*\mu_1) + \wass_1(T^n_*\mu_1,\mu_0) \\
  &= \wass_1(T_{1*}^n \mu_1,T^n_*\mu_1) + \wass_1(T^n_*\mu_1,T^n_*\mu_0) \\
  &\le \lVert T_1^n-T^n\rVert_{C^0} + \wass_\beta^s(T^n_*\mu_1,T^n_*\mu_0)
\end{align*}
We bound $\lVert T_1^n-T^n\rVert_{C^0}$ by induction: for all $x\in M$,
\begin{align*}
d(T_1^n x, T^nx) &\le d\big(T_1 (T_1^{n-1}x), T_1(T^{n-1} x)\big) + d\big(T_1(T^{n-1}x), T(T^{n-1} x)\big) \\
  &\le A\, d(T_1^{n-1} x, T^{n-1}x) + \lVert T_1-T\rVert_{C_0} \\
  & \le (1+A+\dots +A^{n-1}) \lVert T_1-T\rVert_{C_0} \\
  &\lesssim A^n\lVert T_1-T\rVert_{C_0}
\end{align*}
Plugging this into the previous computation, we get
\begin{align*}
D  &\lesssim A^n \lVert T_1-T\rVert_{C^0} +\frac{D}{2} \\
\frac{D}{2} &\lesssim D^{-p} \lVert T_1-T\rVert_{C^0} \\
D &\lesssim \lVert T_1-T\rVert_{C^0}^{\frac{1}{1+p}}.
\end{align*}
\end{proof}

\subsection{Technical part of the proof: $(\mathscr{R}^\beta_K)_{K}$ and its properties}\label{s:technical}

We start be recalling classical facts on Anosov diffeomorphisms and setting up some notations.

Recall that we assume $T:M\to M$ to be a $C^{1+\alpha}$, topologically mixing Anosov diffeomorphism. We denote by $E^s,E^u$ the \emph{stable} and \emph{unstable} distributions, which are $T$-invariant, of constant dimensions $k_s$, $k_u$, and such that $E_x^s\oplus E_x^u = T_xM$ for all $x\in M$. By choosing an adapted riemannian metric on $M$, we can assume that there exist $\lambda\in(0,1)$ such that for all $x\in M$, $u\in E_x^s$, $v\in E_x^u$ and $n\in\mathbb{N}$:
    \[ \lVert D_x(T^n)(u) \rVert \le \lambda^n \lVert u \rVert, \quad
        \lVert D_x(T^{-n})(v) \rVert \le \lambda^n \lVert v \rVert. \]

The stable and unstable distribution can be integrated into the \emph{stable} and \emph{unstable foliations} $(W^s_x)_x$, $(W^u_x)_x$.

The Riemannian metric on $M$ induces a canonical Riemannian metric on each Grassmanian $G(k,T_xM)$ (the space of all $k$-dimensional subspaces of $T_xM$), we denote the induced distance function by $d_{G(k,T_xM)}$. We can in particular define for all continuous $k$-dimensional distributions $E,F$: 
\[d(E,F) := \max_{x\in M} d_{G(k,T_xM)} (E_x,F_x);\]
the Anosov property in particular implies that $T_*$ acts on $G(k_u,T_xM)$ with $E^u$ as an exponentially attracting fixed point: for some $\eta\in (0,1)$ and some $\varepsilon>0$, for all $x\in M$ and all $E_x\in G(k_u,T_xM)$ such that $d_{G(k_u,T_xM)}(E_x,E^u_x)\le \varepsilon$,
\[ d_{G(k_u,T_{T(x)}M)}(D_xT(E_x),E^u_{T(x)})\le \eta\cdot  d_{G(k_u,T_xM)}(E_x,E^u_x). \]
We fix a field of open cones $(C_x)_{x\in M}$ in the tangent bundle $TM$ containing $E_u$, strongly preserved by $T$ (meaning $\overline{T(C_x)}\subset C_{T(x)}$ for all $x\in M$) and in such a way that the above contraction holds on each $U_x := \{E\in G(k_u,T_xM) \mid E\subset C_x\}$.

To compare supspaces of $T_xM$ at different nearby base points $x$, we simply fix a finite smooth atlas of $M$; we say that a $k$-dimension distribution $E$ is $\beta$-Hölder if for some constant $K$, in all charts in the fixed atlas, $E$ is $\beta$-Hölder with that constant $K$. The exponent $\beta$ does not depend on the choice of atlas, but the constant does.

There exist some $\alpha_0\in(0,\alpha)$ such that:
\begin{itemize}
\item $E^s,E^u$ are $\alpha_0$-Hölder continuous, and in particular the leafs of $W^s$, $W^u$ are $C^{1+\alpha_0}$ submanifolds;
\item the holonomies $\pi_s^{N_1\to N_2}$ and $\pi_u^{N_1\to N_2}$ of $W^s$, $W^u$ between  $C^{1,\alpha_0}$ transversals  $N_1$, $N_2$ to $E^s$ (or $E^u$) are $\alpha_0$-Hölder and absolutely continuous with $\alpha_0$-Hölder jacobian; moreover the $\alpha_0$-Hölder constants can be chosen uniform over all $(N_1,N_2)$ that are bounded in diameter, are written as graphs of bounded $C^{1,\alpha_0}$ norm, and are at bounded distance one from the other along $W^s$ (or $W^u$).
\end{itemize}
While $\alpha_0$ depends on $T$, it does so only through a few parameters ($\alpha$, the $\alpha$-Hölder constants of $E^s$ and $E^u$, $\lambda$, a uniform bound on $DT$, etc.) and it can be chosen such that every $C^{1+\alpha}$ maps $T_1$ that is sufficiently close to $T$ in the $C^1$ topology shares the same $\alpha_0$.

From now on we fix any $\beta\in(0,\alpha_0)$.

\subsubsection{Définition of $(\mathscr{R}^\beta_K)_{K}$}\label{s:defiR}

While elements of $(\mathscr{R}^\beta_K)_{K}$ are probability measures, their crucial defining property is a relation with some foliation, which we thus have to define beforehand. We will need to define a regularity constant to control certain submanifolds; to this end we will again use a specified atlas, but with some additionnal features.

We consider smooth charts $\varphi:U_\varphi\subset M \mapsto V_\varphi\subset \mathbb{R}^k$ (where $k$ is the dimension of $M$) whose range $V_\varphi$ are of the form $B^{k_u}(0,1)\times B^{k_s}(0,1)\subset \mathbb{R}^k$, where $B^\bullet(p,r)$ denotes the ball of center $p$ and radius $r$ in a factor $\mathbb{R}^\bullet$ of $\mathbb{R}^k$, with respect to the canonical Euclidean metric. A typical point of $V$ will be denoted by $(p,q)$, with $p\in B^{k_u}(0,1), q\in  B^{k_s}(0,1)$. We denote by $x_\varphi := \varphi^{-1}(0,0)$ the ``center'' of the chart, and we request that the ``slices'' $\varphi^{-1}(B^{k_u}(0,1)\times \{q\})$ are $C^1$-close to $W^u_{\varphi^{-1}(0,q)}$, in particular their tangent space shall be contained in the cone field $C$ at all point; and similarly the $\varphi^{-1}(\{p\}\times B^{k_s}(0,1))$ shall be $C^1$-close to $W^s_{\varphi^{-1}(p,0)}$, in particular their tangent space shall avoid the closure $\bar C=(\bar C_x)_x)$ of the cone field at all point. We can further assume (taking the domains $U_\varphi$ somewhat tall in the stable direction and thin in the unstable direction) that the image $\varphi_*C$ of the cone field is very narrow; in particular that  every $k_u$-dimensional $C^1$ submanifold $N$ whose tangent space is everywhere contained in $C$ and that meets $\mathrm{Core}_\varphi := \varphi^{-1}(B^{k_u}(0,1/2)\times B^{k_s}(0,1/2))$ at some point $x$, can be written locally around $x$ as $\varphi^{-1}(G)$ where $G$ is a graph of a map $B^{k_u}(0,1/2)\to B^{k_s}(0,1)$. In particular, every foliation whose tangent distribution is everywhere contained in $C$ admits a local trivialization that cover the core of $\varphi$.

From now on we fix a finite atlas of such charts, where the $\mathrm{Core}_\varphi$ cover $M$.

\begin{defi}
Let us say that a $k_u$-dimensional foliation $W$ of $M$ is $K$-\emph{adapated} (to $T,C,\beta$) when the leafs of $W$ are $C^{1+\beta}$ and their tangent space is inside $c$ at all point, and in each chart of the specified atlas the leafs are graphs of $\beta$-Hölder with constant at most $K$. 
\end{defi}

\begin{defi}
Let $\mathscr{R}_K^\beta$ be the set of probability measure $\mu$ on $M$ such that there exist a $K$-adapted foliation $W$ with respect to which $\mu$ has absolutely continuous local disintegrations with positive $\beta$-Hölder densities, and each of those densities $\rho$ satisfy 
\begin{equation}
\frac{\rho(x)}{\rho(y)} \le e^{K d_W(x,y)^\beta}
\label{eq:log-holder}
\end{equation}
where $d_W$ denotes the distance along the leafs of $W$. We say that $K$ is a \emph{log constant} of \emph{log $\beta$-Hölder constant} for $\rho$, or for $\mu$. 
\end{defi}

Assume $U_1,U_2$ are two intersecting open sets on which $W$ can be trivialized and $N_1,N_2$ are transversals; i.e. $U_i = \cup_{y\in N_i} W_y\cap U_i \simeq N_i\times B^{k_u}(0,1)$ and consider  $x\in U_1\cap U_2$. Then there are two different densities $\rho_1$, $\rho_2$ corresponding to the disintegrations of the restrictions of $\mu$ to $U_1$, $U_2$ with respect to $N_1$ and $N_2$; but this densities are proportional to each other along leaves where they are both defined. Indeed, denoting by $p_i:U_i\to N_i$ the holonomic projection to the transversal, $\rho_1(x)/\rho_2(x)$ is the Radon derivative of $\pi^{N_1\to N_2}_*p_{1*}\mu$ with respect to $p_{2*}\mu$ at $y=p_2(x)$. In particular, equation \eqref{eq:log-holder} holds for both $\rho_i$ simultaneously, with the same constant.

\subsubsection{Proof of Proposition \ref{p:technical}}

We consider $\beta\in(0,\alpha_0)$ and the above specified atlas fixed.
For each $r>0$ and $x\in M$, set 
\[B_\times(x,r) = \{y\in M \mid \exists z\in M, d_s(x,z)<r, d_u(z,y)<r\}.\]
This ``product balls'' form a basis of the topology of $M$, and we can
find $\delta_-,\delta_+>0$ such that for every chart $\varphi$, 
$B_\times(x_\varphi,\delta_-)\subset \mathrm{Core}_\varphi\subset U_\varphi \subset B_\times(x_\varphi,\delta_+)$. Up to refining the atlas, we can moreover ensure that $\delta_+$ is small enough that both the stable and unstable foliations of $T$ are topologically trivial at scale $\delta_+$ (i.e. the $B_\times(x,\delta_+)$ are domains of trivailizing charts). Since $T$ is assumed to be topologically mixing, its stable leafs are dense and we can find $\tilde L_0$ such that for every pair of charts $\varphi,\psi$, their is a stable path of length at most $\tilde L_0$ from $\varphi^{-1}(0,0)$ to a point $\delta_-/100$-close to $x_\psi$ in the distance $d_u$. We set $L_0=\tilde L_0+2\delta_+$.

To see that the SRB measure $\mu_0$ lies in some $\mathscr{R}_{K}^\beta$, it suffices to use the usual expression for the densities of its local disintegrations as an infinite product involving the unstable jacobian of $T$, see section 9.3 in \cite{Barreira2007book}.

By definition, we immediately get $\mathscr{R}_K^\beta\subset\mathscr{R}_{K'}^\beta$ whenever $K'>K$. Moreover if $\mu\in\mathscr{R}_K^\beta$ and $\rho\in \holc_W^\beta$ is a positive density with respect to $\mu$, then $\rho\mu$ has absolutely continuous disintegrations with respect to the same $K$-adapted foliation $W$ as $\mu$; if $\rho_1$ is a local leaf density of $\mu$, then the local leaf density of $\rho\mu$ is proportional to $\rho\rho_1$, hence $\rho\mu\in \mathscr{R}_{K'}^\beta$ for $K'=K +\lVert\log \rho\rVert_{W,\beta}$.

As in Section \ref{s:expanding}, there is $H>0$ and $\lambda'\in (0,1)$ such that whenever $\mu\in \mathscr{R}_K^\beta$, $T_*\mu\in \mathscr{R}_{\lambda'(K+H)}^\beta$; the constant $\lambda'$ is $\beta$th power of the contraction factor of $T$ on every submanifold whose tangent space is contained in $C$ at all point, and $H$ accounts for the unstable jacobian of $T$ and the way $T$ and change of charts distort $\beta$-Hölder graphs. We can thus define $K_0$ and $n_0$ by
\[K_0 = 2\frac{\lambda'}{1-\lambda'}H, \qquad 2\lambda'^{n_0} \le \frac{\lambda'}{1-\lambda'}\]
to obtain that  $T_*(\mathscr{R}_{K_0}^\beta)\subset \mathscr{R}_{K_0}^\beta$ and $T^{n_0}_*(\mathscr{R}_{2K}^\beta)\subset \mathscr{R}_{K}^\beta$ for all $K\ge K_0$.

We now prove the core coupling property, item \ref{enumi:tech4}. Let $\mu_1,\mu_2\in \mathscr{R}_{K_0}^\beta$ and denote by $W^1,W^2$ their $K_0$-adapted foliations. Denote by $N$ the number of charts in the chosen atlas and recall that the $\mathrm{Core}_\varphi:=\varphi^{-1}(B^{k_u}(0,1/2)\times B^{k_s}(0,1/2))$ cover $M$. We can thus find two charts $\varphi_1,\varphi_2$ such that $\mu_i(\mathrm{Core}_{\varphi_i})\ge 1/N$ for both $i$; and since the cores are covered by trivializing charts of the $W^i$, we can locally disintegrate restriction $\tilde\mu_i$ of the $\mu_i$ through the holonomic projections $p_i$ to local stable leaves $D_i\subset W^s_{x_{\varphi_i}}$:
\[\mu_i = \int_{D_i} \rho_{i,x}(y) \dd\mathrm{vol}_{W^i_x}(y) \dd\eta_i(x)\] 
where $\eta_i$ are positive measures on $D_i$ of mass at least $1/N$ and $\rho_i$ are $\beta$-Hölder with log constant at most $K_0$. Consider a $1/N$-partial coupling $\tilde\gamma$  of $(\eta_1,\eta_2)$, i.e. $\gamma$ is a positive measure of mass $1/N$ on $M\times M$ with marginals $\le \eta_i$; for concreteness, let us take $\tilde\gamma=\frac1N \eta_1\otimes \eta_2$. For each pair of leafs $W^1_x, W^2_y$ we construct a partial coupling of the $\rho_i \dd\mathrm{vol}_{W^i_x}$ as follows. Given $h>0$, Let $f=f_{x,h}:W^1_x\to \mathbb{R}_+$ be defined by $f(z) = \max(0,h-\frac{10 h}{\delta_-}d_{W^1_x}(x,z))$. Let $g=g_{y,h}$ be defined as the composition of $f$ with the holonomy $\pi^s_{x,y}$ of $W^s$, along a curve of length at most $L_0$.   By choosing $h\le h_0$ with $h_0$ small enough, we can ensure that both $\rho_{1,x}-f_{x,h}$ and $\rho_{1,x}-g_{y,h}$ are positive and $\beta$-Hölder with log constant at most $2K_0$ (indeed, $f$ is Lipschitz with constant only depending on $\delta_-$, hence on $T$, and $g$ is $\alpha_0$-Hölder with constant depending further on $L_0$ and the stable holonomy Hölder constant, thus still only depending on $T$). Now all considered submanifold in the foliations are $C^{1,\beta}$ with bounded Hölder constant in the specified charts, so that there exist $\tilde\tau>0$ such that the function $f_{x,h_0}$ has integral at least $\tilde\tau$ with respect to $\mathrm{vol}_{W^1_x}$. By choosing $h=h_{x}\le h_0$ depending on $x$, we can make $f_{x,h}$ of integral exactly $\tilde\tau$. finally, our partial coupling of $\mu_1,\mu_2$ is given by
\[\gamma = \int (\mathrm{Id},\pi^s_{x,y})_* \big(f_{x,h_{x}} \mathrm{vol}_{W^1_x}\big) \dd\tilde\gamma(x,y).\]
It has mass $\tau :=\tilde\tau/N$, and its marginals are $\le\mu_i$; moreover $\mu'_i := \mu_i-p_{i*}\gamma$ have local disintegrations with respect to $W^i$ that are $\beta$-Hölder with log constant at most $2K_0$, hence $\mu'_i\in\mathscr{R}_{2K_0}^\beta$.

\subsubsection{Final remarks}

A similar method can be used to handle hyperbolic attractors instead of Anosov maps, but Theorem \ref{t:central} does not hold anymore in this setting (stable leafs are no longer dense, and most measures to be considered will no longer be at finite $\wass_s^\beta$ distance, for any $\beta$). What can be done is to pair iterates of measures $T^{n_1}_*\mu_1, T^{n_1}_*\mu_2$ where $n_1$ is chosen to ensure that a significant part of the mass of $\mu_1$, $\mu_2$ are at small stable distance one from the other.

If we relax uniform hyperbolicity, we can hope to keep weaker results in the same vein. In the case of lack of uniform hyperbolicity in the stable direction, it might be necessary to use $\omega\circ d_s$ instead of $d_s^\beta$, where $\omega$ is a concave function increasing quicker than any Hölder function at $0$. In the case of lack of uniform hyperbolicity in the unstable direction, it might be necessary to use densities and observables with an adapted regularity. In both cases, weaker speed for the decay of correlation is of course expected.

It might be possible to adapt the present method to maps of flows with discontinuities, given it relies on a relatively flexible coupling argument: pair together along the stable direction part of any two measures that are regular enough in the unstable direction, then apply the dynamics long enough to recover the initial regularity for the remaining, uncoupled parts. Rince and repeat.

\bibliographystyle{smfalpha}
\bibliography{SRB}

\providecommand{\bysame}{\leavevmode ---\ }
\providecommand{\og}{``}
\providecommand{\fg}{''}
\providecommand{\smfandname}{et}
\providecommand{\smfedsname}{\'eds.}
\providecommand{\smfedname}{\'ed.}
\providecommand{\smfmastersthesisname}{M\'emoire}
\providecommand{\smfphdthesisname}{Th\`ese}
\begin{thebibliography}{AdCJ04}

\bibitem[AdCJ04]{deCastro2004attractor}
{\scshape A.~Armando~de Castro~J\'{u}nior} -- {\og Fast mixing for attractors
  with a mostly contracting central direction\fg}, \emph{Ergodic Theory Dynam.
  Systems} \textbf{24} (2004), no.~1, p.~17--44.

\bibitem[AS12]{Alves2012statistical}
{\scshape J.~F. Alves {\normalfont \smfandname} M.~Soufi} -- {\og Statistical
  stability and limit laws for {R}ovella maps\fg}, \emph{Nonlinearity}
  \textbf{25} (2012), no.~12, p.~3527--3552.

\bibitem[BP07]{Barreira2007book}
{\scshape L.~Barreira {\normalfont \smfandname} Y.~Pesin} -- \emph{Nonuniform
  hyperbolicity}, Encyclopedia of Mathematics and its Applications, vol. 115,
  Cambridge University Press, Cambridge, 2007, Dynamics of systems with nonzero
  Lyapunov exponents.

\bibitem[BV96]{Baladi1996stochastic}
{\scshape V.~Baladi {\normalfont \smfandname} M.~Viana} -- {\og Strong
  stochastic stability and rate of mixing for unimodal maps\fg}, \emph{Ann.
  Sci. \'{E}cole Norm. Sup. (4)} \textbf{29} (1996), no.~4, p.~483--517.

\bibitem[BY00]{Benedicks2000Henon}
{\scshape M.~Benedicks {\normalfont \smfandname} L.-S. Young} -- {\og Markov
  extensions and decay of correlations for certain {H}\'{e}non maps\fg}, no.
  261, 2000, G\'{e}om\'{e}trie complexe et syst\`emes dynamiques (Orsay, 1995),
  p.~xi, 13--56.

\bibitem[Dol98]{Dolgopyat1998flows}
{\scshape D.~Dolgopyat} -- {\og On decay of correlations in {A}nosov flows\fg},
  \emph{Ann. of Math. (2)} \textbf{147} (1998), no.~2, p.~357--390.

\bibitem[DSL16]{DeSimoi2016mostly}
{\scshape J.~De~Simoi {\normalfont \smfandname} C.~Liverani} -- {\og
  Statistical properties of mostly contracting fast-slow partially hyperbolic
  systems\fg}, \emph{Invent. Math.} \textbf{206} (2016), no.~1, p.~147--227.

\bibitem[KKM19]{Korepanov2019coupling}
{\scshape A.~Korepanov, Z.~Kosloff {\normalfont \smfandname} I.~Melbourne} --
  {\og Explicit coupling argument for non-uniformly hyperbolic
  transformations\fg}, \emph{Proc. Roy. Soc. Edinburgh Sect. A} \textbf{149}
  (2019), no.~1, p.~101--130.

\bibitem[Liv95]{Liverani1995decay}
{\scshape C.~Liverani} -- {\og Decay of correlations\fg}, \emph{Ann. of Math.
  (2)} \textbf{142} (1995), no.~2, p.~239--301.

\bibitem[Var08]{Varandas2008correlation}
{\scshape P.~Varandas} -- {\og Correlation decay and recurrence asymptotics for
  some robust nonuniformly hyperbolic maps\fg}, \emph{J. Stat. Phys.}
  \textbf{133} (2008), no.~5, p.~813--839.

\bibitem[You98]{Young1998statistical}
{\scshape L.-S. Young} -- {\og Statistical properties of dynamical systems with
  some hyperbolicity\fg}, \emph{Ann. of Math. (2)} \textbf{147} (1998), no.~3,
  p.~585--650.

\end{thebibliography}
\end{document}